\documentclass[12pt]{article}
\usepackage[a4paper,left=2.54cm,right=2.54cm,top=2.54cm,bottom=2.54cm]{geometry}
\usepackage{setspace}   
\usepackage{amsmath,amssymb,amsthm,graphicx,mathrsfs}
\usepackage{enumerate}
\usepackage{hyperref}
\usepackage[capitalise]{cleveref}
\usepackage{color,xcolor}
\usepackage[numbers,sort&compress]{natbib}
\allowdisplaybreaks
\numberwithin{equation}{section}

\newtheorem{theorem}{Theorem}[section]
\newtheorem{lemma}[theorem]{Lemma}
\newtheorem{definition}[theorem]{Definition}

\newtheorem{proposition}[theorem]{Proposition}

\newcommand{\benu}{\begin{enumerate}}
	\newcommand{\beqa}{\begin{eqnarray}}
		\newcommand{\beqan}{\begin{eqnarray*}}
			\newcommand{\eay}{\end{array}}
		\newcommand{\edm}{\end{displaymath}}
	\newcommand{\eenu}{\end{enumerate}}

\newcommand{\eeq}{\end{equation}}
\newcommand{\eeqa}{\end{eqnarray}}
\newcommand{\eeqan}{\end{eqnarray*}}

\newcommand{\bea}{\begin{array}{cc}}
\newcommand{\ena}{\end{array}}

\begin{document}
\pagenumbering{arabic} \setcounter{page}{1}

\begin{center}
{\Large \bf
Derivations and local  derivations on Euclidean Lie algebras}

\vspace{0.3in}
Lingen Ding$^{1}$

$^{1}$\ School of Mathematics, Foshan  University, Foshan  528000, Guangdong, China\\
Email:  lingending@fosu.edu.cn,
\end{center}
\vspace{3mm}

\begin{abstract}\normalsize
The present paper is devoted to studying derivations and local derivations on the
Euclidean Lie algebras
 $\mathfrak{e}(n)$. We give a complete desrciption of the derivation algebra of \(\mathfrak{e}(n)\) for \(n\geq 4\). Furthermore, we prove that 
every 
local derivations on the $\mathfrak{e}(n)$ is 
a derivation.

\end{abstract}
\vspace{3mm}



\section{Introduction}
The notion of a local derivation was introduced and studied independently by Kadison\cite{Kad} and Larson and Sourour\cite{LS} in 1990 for Banach (and associative) algebras. Given an algebra $A$, a linear map $\Delta : A \to A$ is called a \emph{local derivation} if for every $x \in A$ there exists a derivation $D_x : A \to A$ (depending on $x$) such that $\Delta(x) = D_x(x)$. Local derivations reflect an interesting class of local-to-global phenomena on various algebraic structures; see \cite{AEK,AK1,AK2,AKh, Joh,LZ,Sem}  and the references therein.

Over the past decade, considerable attention has been devoted to local derivations on Lie (super)algebras. A central theme in this line of research is to determine whether all local derivations are indeed derivations for a given Lie (super)algebra. This property has been established for finite-dimensional semisimple Lie algebras\cite{AK2}, Borel subalgebras of finite-dimensional simple Lie algebras\cite{YC}, the Witt algebra \cite{CZZ}, the Witt algebra over fields of prime characteristic\cite{Yao}, the $W(2,2)$ algebra\cite{WGL}, the super Virasoro algebra\cite{WGLC}, the conformal Galilei algebra \cite{AB}, the Schr\"odinger algebras\cite{WT,AY} , and solvable Lie algebras of maximal rank\cite{KKY}. On the other hand, certain Lie algebras admit \emph{pure} local derivations-that is, local derivations which fail to be derivations. Examples include finite-dimensional filiform Lie algebras\cite{AK2} and the $1$-spatial ageing algebra $\mathfrak{age}(1)$\cite{LW}.

Let $E(n)$ denote the group of Euclidean motions in $\mathbb{R}^n$; this is the noncompact semidirect product group $SO(n) \ltimes \mathbb{R}^n$. The complexification of its Lie algebra, denoted by $\mathfrak{e}(n)$, admits a basis
\[
\{E_{i,j}, H_k \mid 1 \le i < j \le n,\ 1 \le k \le n\}
\]
with nonzero commutation relations given by
\[
[E_{i,j}, E_{j,k}] = E_{i,k}, \quad
[E_{i,j}, H_j] = H_i, \quad
[E_{i,j}, H_i] = -H_j,
\]
where we set $E_{i,j} = -E_{j,i}$ throughout. The $2$-local derivations and bi-derivations of $\mathfrak{e}(3)$ were studied in\cite{Ab}. For $n \ge 4$, the problem of characterizing local derivations on $\mathfrak{e}(n)$ has remained open.

In this paper, we determine all derivations and local derivations on the Euclidean Lie algebra $\mathfrak{e}(n)$ for $n \ge 4$. Our main results can be summarized as follows:
\begin{itemize}
  \item We give a complete description of $\operatorname{Der}(\mathfrak{e}(n))$ (Theorem~\ref{thm:Der});
  \item We prove that every local derivation on $\mathfrak{e}(n)$ is a derivation (Theorem~\ref{thm:locDer}).
\end{itemize}

The paper is organized as follows. In Section~\ref{sec:prelim}, we collect basic definitions and establish notation. In Section~\ref{Sec3 Der}, we characterize the derivations of $\mathfrak{e}(n)$. In Section~\ref{Sec4 locDer}, we show that every local derivation on $\mathfrak{e}(n)$ is a derivation.

Throughout this paper, $\mathbb{C}$ denotes the field of complex numbers, and all vector spaces and algebras are understood to be over $\mathbb{C}$ unless stated otherwise.


\section{ Preliminaries}\label{sec:prelim}

In this section we recall some definitions, symbols and notations for later use in this paper.

The group \( E(n) \) of Euclidean motions in \( \mathbb{R}^n \) is the noncompact semidirect product group \( SO(n) \ltimes \mathbb{R}^n  \). The complexification of its Lie algebra  called \( \mathfrak{e}(n) \) , which admits a basis \( \{E_{i,j}, H_k \mid 1 \leq i < j\leq n,\, 1\leq k\leq n\} \) with non-zero commutation relations given by
\[
[E_{i,j}, E_{j,k}] = E_{i,k}, \quad [E_{i,j}, H_j] = H_i, \quad [E_{i,j}, H_i] = -H_j,
\]
assuming \( E_{i,j} = -E_{j,i} \).

In fact, the matrix realization of the Euclidean Lie algebra \( \mathfrak{e}(n) \) can be implemented by the following block matrix form:
\[
\begin{pmatrix}
    &   &   & x_1 \\
    & A &   & \vdots \\
    &   &   & x_n \\
    0 & 0 & \dots & 0
\end{pmatrix}
\]
where the matrix \( A \) is an $n\times n$ skew-symmetric matrix. In this realization, 
\[
E_{i,j} = e_{i,j} - e_{j,i}, \quad 1 \leq i \neq j \leq n, \qquad 
H_k = e_{k,n+1}, \quad 1 \leq k \leq n
\]
with the matrix units \( e_{i,j} \). The Euclidean Lie algebra \( \mathfrak{e}(n) \) can be viewed as a semidirect product \( 
\mathfrak{e}(n) = \mathfrak{so}(n)\ltimes H(n)  \) of two subalgebras:
the Lie algebra
\( \mathfrak{so}(n) = \operatorname{span}\{ E_{i,j}\mid 1 \leq i<j\leq n\} \) and 
the abelian subalgebra 
\(
H(n)= \operatorname{span}\{H_{k}\mid 1 \leq k\leq n\} \). 
Define \(p_0 :\mathfrak{e}(n)\longrightarrow \mathfrak{so}(n)\) and \(p_1 :\mathfrak{e}(n)\longrightarrow H(n)\) as the projections onto the rotation and translation parts, respectively, in the decomposition \( 
\mathfrak{e}(n) = \mathfrak{so}(n)\ltimes H(n)  \).

A derivation on a Lie algebra \( L \) is a linear map \( D : L \to L \) which satisfies the Leibniz rule:
\[
D([x, y]) = [D(x), y] + [x, D(y)] \quad \text{for any } x, y \in L.
\]
The set of all derivations of \( L \) is denoted by \( \operatorname{Der}(L) \) and is a Lie algebra with respect to the commutation operation.

For \( x \in L \), the map \( \operatorname{ad}x : L \to L \), defined by \( \operatorname{ad}x(z) = [x, z] \) for all \( z\in L\), is a derivation of \( L \), and derivations of this form are called inner derivations. The set of all inner derivations of \( L \), denoted by \( \operatorname{IDer}(L) \), is an ideal in \( \operatorname{Der}(L) \).
\begin{lemma}\label{lem:2.1}
    Let \(D\) be any derivation on the Euclidean Lie algebra \(\mathfrak{e}(n) = \mathfrak{so}(n) \ltimes H(n)\). Define the linear maps \(D_0, D_1: \mathfrak{e}(n) \to \mathfrak{e}(n)\) by
    \[
    D_0 = p_0 \circ D \circ p_0 + p_1 \circ D \circ p_1, \quad D_1 = p_0 \circ D \circ p_1 + p_1 \circ D \circ p_0,
    \]
    where \(p_0\) and \(p_1\) are the projections onto the subspaces \(\mathfrak{so}(n)\) and \(H(n)\), respectively. Then \(D=D_0+D_1\) and 
    \(D_0,\,D_1\) are also derivations on \(\mathfrak{e}(n)\).
\end{lemma}
\begin{proof}
The identity $D = D_0 + D_1$ follows immediately from
$p_0 + p_1 = \text{Id}_{\mathfrak{e}(n)}$.

It remains to verify the Leibniz rule for $D_0$ and $D_1$.
For any  elements $x = x_0 + x_1$, $y = y_0 + y_1$
with $x_0, y_0 \in \mathfrak{so}(n)$ and $x_1, y_1 \in H(n)$,
we have the Lie bracket decomposition
\begin{equation}\label{eq:bracket-decomp}
    [x, y] = [x_0, y_0] + [x_1, y_0] + [x_0, y_1],
\end{equation}
where $[x_0, y_0] \in \mathfrak{so}(n)$ and $[x_1, y_0], [x_0, y_1] \in H(n)$.
(Note that $H(n)$ is an abelian ideal of $\mathfrak{e}(n)$,
so $[x_1, y_1] = 0$.)

By hypothesis, $D$ is a derivation on $\mathfrak{e}(n)$,
and the projections satisfy the compatibility conditions:
\begin{equation}\label{eq:compatibility}
    \begin{aligned}
        &D_0(\mathfrak{so}(n))) \subseteq \mathfrak{so}(n),     &&D_0(H(n))   \subseteq H(n),   \\
        &D_1(\mathfrak{so}(n)) \subseteq H(n),      &&D_1(H(n))  \subseteq \mathfrak{so}(n).
    \end{aligned}
\end{equation}

We now verify that $D_0$ satisfies the Leibniz rule.
Using~\eqref{eq:bracket-decomp} and~\eqref{eq:compatibility},
\begin{align}
    D_0([x, y])
    &= p_0\bigl( D(p_0([x, y])) \bigr)
       + p_1\bigl( D(p_1([x, y])) \bigr) \notag \\[4pt]
    &= p_0\bigl( D([x_0, y_0]) \bigr)
       + p_1\Bigl( D\bigl( [x_1, y_0] + [x_0, y_1] \bigr) \Bigr) \notag \\[4pt]
    &= p_0\bigl( [D(x_0), y_0] + [x_0, D(y_0)] \bigr)  + p_1\bigl( [D(x_1), y_0] + [x_1, D(y_0)] \bigr) \notag \\
    &\quad + p_1\bigl( [D(x_0), y_1] + [x_0, D(y_1)] \bigr) \label{eq:D0-line5} \\[4pt]
    &= [D_0(x_0), y_0] + [x_0, D_0(y_0)]  + [D_0(x_1), y_0] + [x_1, D_0(y_0)] \notag \\
    &\quad + [D_0(x_0), y_1] + [x_0, D_0(y_1)]\notag  \\[4pt]
    &= [D_0(x_0), y] + [x_0, D_0(y)]
       + [D_0(x_1), y] + [x_1, D_0(y)] \notag \\[4pt]
    &= [D_0(x), y] + [x, D_0(y)]. \notag
\end{align}
Thus $D_0$ is a derivation on $\mathfrak{e}(n)$.

A completely analogous computation shows that $D_1$
is also a derivation on $\mathfrak{e}(n)$. 
\end{proof}

\begin{definition}
A linear map \( \Delta \) is called a \emph{local derivation} if for any \( x \in L \), there exists a derivation \( D_x : L \to L \) (depending on \( x \)) such that \( \Delta(x) = D_x(x) \). The set of all local derivations on \( L \) is denoted by \( \operatorname{LocDer}(L) \).
\end{definition}
\begin{lemma}
Let \(\Delta\) be a local derivation on the Euclidean algebra \(\mathfrak{e}(n)\). Defined  \(\Delta_{0,0}:\mathfrak{e}(n)\longrightarrow \mathfrak{e}(n)\) as 
    \[
    \Delta_{0,0}=p_0\circ \Delta \circ p_0
    \]
    Then \(\Delta_{0,0}\mid_{\mathfrak{so}(n)}\) is a local derivation on 
\(\mathfrak{so}(n)\).
     \end{lemma}
\begin{proof}
   Let \(\Delta\) be a local derivation on the Euclidean Lie algebra \(\mathfrak{e}(n)\). Then the restriction \(\Delta_{0,0}|_{\mathfrak{so}(n)}\) defines a linear map from \(\mathfrak{so}(n)\) to itself. For any \(x \in \mathfrak{so}(n)\), we have
\[
\Delta_{0,0}(x) = p_0 \circ \Delta(x) = p_0 \circ D_x(x) = (D_x)_0(x).
\]
Moreover, \((D_x)_0|_{\mathfrak{so}(n)}\) is a derivation on \(\mathfrak{so}(n)\), which implies that \(\Delta_{0,0}|_{\mathfrak{so}(n)}\)  is a local derivation on \(\mathfrak{so}(n)\).
\end{proof}


\section{Derivations of $\mathfrak{e}(n)$}\label{Sec3 Der}

In this section, we determine the derivations of 
 the Euclidean Lie algebra $\mathfrak{e}(n)$. 

Define
\( \delta \) to be a linear map from \( \mathfrak{e}(n)\) to itself by 
\[
\delta\mid_{\mathfrak{so}(n)} =0, \ \delta\mid_{H(n)} =\text{Id}_{H(n)}.\]
\begin{lemma}\label{D0}
    Let \(D\) be any derivation on \(
\mathfrak{e}(n)
    \). Then \(D_0=\text{ad}X+\lambda \delta\) for some \(X\in \mathfrak{so}(n),\lambda\in \mathbb{C}.\)
\end{lemma}
\begin{proof}
Clearly, \(\delta\) is a derivation on \(\mathfrak{e}(n)\) and \(\delta_0=\delta\).

For any \(D\in \text{Der}(\mathfrak{e}(n))\), \(D_0\mid_{\mathfrak{so}(n)}\) is a derivation on \(\mathfrak{so}(n)\). As $n\geq 4$, $\mathfrak{so}(n)$ is a semisimple Lie algebra. Then \(D_0\mid_{\mathfrak{so}(n)}\) is an inner derivation on \(\mathfrak{so}(n)\), i.e., \(D_0\mid_{\mathfrak{so}(n)}=\text{ad}\,X\) for some \(X\in \mathfrak{so}(n)\).

Set \(D_0'=D_0-\text{ad}\,X\). Then \(D_0'(x)=0\) for all \(x\in \mathfrak{so}(n)\). For any \(1\leq i\leq n\), we can write 
\[D_0'(H_i)=\sum_{1\leq j\leq n}d_{i,j}H_j, \,\text{where}\ d_{i,j}\in \mathbb{C}.\]
Choose distinct indices \(i,j\) with \(1 \leq i \neq j \leq n\). For any \(k \notin \{i,j\}\), we have \([E_{k,j}, H_i] = 0\), and therefore
\[
0 = D_0'([E_{k,j}, H_i]) = [E_{k,j}, D_0'(H_i)] = \left[ E_{k,j}, \sum_{1 \leq l \leq n} d_{i,l} H_l \right] = d_{i,j} H_k - d_{i,k} H_j.
\]
It follows that \(d_{i,j} = 0\) for all distinct \(i, j\). Then \(D_0'(H_i)=d_{i,i}H_i\) for all \(i=1,\ldots,n\).

Now using the relation \([E_{i,j}, H_j] = H_i\), we obtain
\[
D_0'(H_i) = D_0'([E_{i,j}, H_j]) = [E_{i,j}, D_0'(H_j)]  = [E_{i,j}, d_{j,j} H_j] = d_{j,j} H_i.
\]
Since we also have \(D_0'(H_i) = d_{i,i} H_i\), it follows that \(d_{i,i} = d_{j,j}\) for all \(1 \leq i\neq j\leq n\). Hence \(D_0' = \lambda \delta\) for some \(\lambda \in \mathbb{C}\), and consequently \(D_0 = \operatorname{ad} X + \lambda \delta\).

\end{proof}

\begin{lemma}\label{D1;Hn}
    Let \(D\) be any derivation on \(
\mathfrak{e}(n)
    \). Then \(D_1\mid_{H(n)}=0\).
\end{lemma}
\begin{proof}
  For each \(1\leq k\leq n\), write \[D_1(H_k)=\sum\limits_{1\leq i<j\leq n}c_{i,j}^{(k)}E_{i,j}.\]
    Fix indices \(i_0,j_0\) with \(1 \leq i_0< j_0\leq n\) and \(k\notin \{i_0,j_0\}\).
    Since \(n\geq 4\), we can choose an index \(l\) such that \(1 \leq l\leq n\) and  \(l\notin \{i_0,j_0,k\}\).
Then \([E_{i_0,l},H_k]=0\).
Because $D_1(E_{i_0,l})\in H(n)$ and Lemma \ref{lem:2.1},
we obtain
    \[
[E_{i_0,l},D_1(H_k)]=D_1([E_{i_0,l},H_k])-[D_1(E_{i_0,l}),H_k]=0.
\] 
 By the following equation
    \begin{align}\label{D_1;ij}
&\left[ E_{i_0,j_0},\left[
[D_1(H_k), E_{i_0,l}],H_{l}
\right]\right] \notag\\
&=\left[ E_{i_0,j_0},\left[
\left[\sum\limits_{1\leq i<j\leq n}c_{i,j}^{(k)}E_{i,j}
, E_{i_0,l}\right],H_{l}
\right]\right]\notag \\
&=\left[ E_{i_0,j_0},\left[
\sum\limits_{1\leq i<j\leq n}c_{i,j}^{(k)}\left[E_{i,j}
, E_{i_0,l}\right],H_{l}
\right]\right] \\
&=\left[ E_{i_0,j_0},\left[
\sum\limits_{1\leq i<j\leq n}c_{i,j}^{(k)}
(\delta_{j,i_0}E_{i,l}-\delta_{l,i}E_{i_0,j}-\delta_{i,i_0}E_{j,l}-\delta_{j,l}E_{i,i_0})
,H_{l}
\right]\right]\notag \\
&=\left[ E_{i_0,j_0},
\sum\limits_{1\leq i<i_0\leq n,\,i\neq l}c_{i,i_0}^{(k)}
H_i
-\sum\limits_{1\leq i_0<j\leq n, j\neq l}c_{i_0,j}^{(k)}H_j
\right]\notag \\
&=-c_{i_0,j_0}^{(k)}H_{i_0}=0,\notag
\end{align}
we have  \(c_{i_0,j_0}^{(k)}=0\).

Next, we consider the coefficients \(c_{k,j}^{(k)}\) for \(j>k \).
In view of Lemma \ref{lem:2.1}, we have 
\begin{align}\label{D_1;kj}
D_1(H_k)&=D_1([E_{k,j},H_j])\notag\\
 &=[D_1(E_{k,j}), H_j] + [E_{k,j}, D_1(H_j)]\notag\\
 &=[E_{k,j}, D_1(H_j)]\notag\\
 &=\left[E_{k,j}, \sum\limits_{1\leq s<t\leq n}c_{s,t}^{(j)}E_{s,t}
 \right]\notag\\
 &=\sum\limits_{1\leq s<t\leq n}c_{s,t}^{(j)}
(\delta_{j,s}E_{k,t}-\delta_{t,k}E_{s,j}-
\delta_{j,t}E_{k,s}-
\delta_{k,s}E_{j,t})\\
&=\sum\limits_{1\leq j<t\leq n}c_{j,t}^{(j)}
E_{k,t}-
\sum\limits_{1\leq s<k\leq n}c_{s,k}^{(j)}
E_{s,j}-
\sum\limits_{1\leq s<j\leq n,\, s\neq k}c_{s,j}^{(j)}E_{k,s}\notag\\
&\hspace{6mm}-
\sum\limits_{1\leq k<t\leq n,\ t\neq j}c_{k,t}^{(j)}
E_{j,t}.\notag
\end{align}
From \cref{D_1;kj}, it implies that \(D_1(H_k)\) does not include any term of the form \(c_{k,j}^{(k)}E_{k,j}\); thus \(c_{k,j}^{(k)} = 0\) for all \(j > k\).

Similarly, we show that \(c_{i,k}^{(k)} = 0\) for all \(i < k\). Consequently, \(D_1(H_k) = 0\) for each \(k = 1, \ldots, n\).
    \end{proof}

\begin{proposition}\label{D_1}
    Let \(D\) be any derivation on \(\mathfrak{e}(n)\). Then \(D_1\) is an inner derivation.
\end{proposition}
\begin{proof}
    For any pair \((i,j)\) with \(1\leq i\neq j\leq n\),  we write
\[D_1(E_{i,j})=\sum\limits_{1\leq p\leq n}b_{(i,j)}^{p}H_p.\]
Indeed, \(b_{(i,j)}^{p}=-b_{(j,i)}^{p}\).
Then, for any \(
1\leq s<t\leq n\), we have
\begin{equation}\label{D_1;ijst}
[D_1(E_{i,j}),E_{s,t}]\in \mathbb{C}H_s \oplus\mathbb{C} H_t.
\end{equation}

 Choose \(1\leq k\leq n\) with
 \(k\notin \{i,j\}\), then there exists $l$ such that \(l\notin \{i,j,k\}\). Then  \([E_{i,j},E_{k,l}]=0\) and hence
\begin{align}\label{D_1;ijkl}
    [D_1(E_{i,j}),E_{k,l}]=-[E_{i,j},D_1(E_{k,l})].
\end{align}
Combining \cref{D_1;ijst,D_1;ijkl},  we obtain
\[
[D_1(E_{i,j}),E_{k,l}] =b_{(i,j)}^{k}H_l-b_{(i,j)}^{l}H_k=0,
\]
which implies that \(b_{(i,j)}^{k}=0\) for all \(k\neq i,j\). Consequently,
\(
D_1(E_{i,j})=b_{(i,j)}^{i}H_i+b_{(i,j)}^{j}H_j.
\)

Since \([E_{k,j}, E_{i,j}] = -E_{k,i}\) for any \(k \notin \{i,j\}\), we have
\begin{align}
D_1(E_{k,i}) &= b_{(k,i)}^{k}H_k + b_{(k,i)}^{i}H_i \nonumber\\
&= -[D_1(E_{k,j}), E_{i,j}] - [E_{k,j}, D_1(E_{i,j})] \nonumber\\
&= -[b_{(k,j)}^{k}H_k + b_{(k,j)}^{j}H_j,\, E_{i,j}] - [E_{k,j},\, b_{(i,j)}^{i}H_i + b_{(i,j)}^{j}H_j] \\
&= b_{(k,j)}^{j}H_i - b_{(i,j)}^{j}H_k.\nonumber
\end{align}
Comparing the coefficients of \(H_i\) and \(H_k\) in the expressions for \(D_1(E_{k,i})\), one yields
\begin{align}\label{eq:b}
b_{(k,i)}^{i} = b_{(k,j)}^{j}, \quad b_{(k,i)}^{k} = -b_{(i,j)}^{j} = b_{(j,i)}^{j}.
\end{align}

Let \(Y=-\sum\limits_{2\leq j\leq n}b_{(1,j)}^{1}H_j-b_{(2,1)}^{2}H_1\).
 Assume that \(1<i<j\leq n\). Then for \cref{eq:b}, we have
  \begin{align}\label{ad-ij}
      \text{ad}(Y)(E_{i,j})&=[-b_{(1,i)}^{1}H_i-b_{(1,j)}^{1}H_j,E_{i,j}]\nonumber\\
      &=b_{(1,j)}^{1}H_i-b_{(1,i)}^{1}H_j\nonumber\\
      &=b_{(i,j)}^{i}H_i-b_{(j,i)}^{j}H_j\\
&=b_{(i,j)}^{i}H_i+b_{(i,j)}^{j}H_j=D_1(E_{i,j}).\nonumber
  \end{align} 
    Suppose that \(1=i<j\leq n\). Then for \cref{eq:b}, we obtain 
   \begin{align}\label{ad;1j}
      \text{ad}(Y)(E_{1,j})&=[-b_{(2,1)}^{2}H_1
    -b_{(1,j)}^{1}H_j,E_{1,j}]\nonumber\\
      &=b_{(1,j)}^{1}H_1-b_{(2,1)}^{2}H_j\\
&=b_{(1,j)}^{1}H_1+b_{(1,j)}^{j}H_j=D_1(E_{1,j}).\nonumber
  \end{align}  
  It then follows from  Lemma \ref{D1;Hn} and \cref{ad-ij,ad;1j}  that 
  \[D_1=\text{ad}\,Y,\]
   thereby completing the proof.
\end{proof}
 Together with Lemma \ref{D0}, Lemma \ref{D1;Hn}, and Proposition \ref{D_1}, we obtain the following theorem, which is the main result of this section.

\begin{theorem}\label{Der-en}
\label{thm:Der}
\( \operatorname{Der}(\mathfrak{e}(n)) = \operatorname{IDer}(\mathfrak{e}(n)) \oplus \mathbb{C} \delta \).
\end{theorem}

\section{Local derivations on \(\mathfrak{e}(n)\)}\label{Sec4 locDer}
In this section, we determine the local derivations of $\mathfrak{e}(n)$.

    For any local derivation $\Delta$ on $\mathfrak{e}(n)$, by Theorem \ref{Der-en}, we have 
    \begin{align*}
        \Delta(x)=[t_x,x]+\lambda_x\delta(x)\quad \text{for all} \; x\in \mathfrak{e}(n).
    \end{align*}
Then we have the following Lemma for later use.
\begin{lemma}\label{lem:4.1}
    Let \(\Delta\) be a local derivation on \(\mathfrak{e}(n)\). Then
    \begin{enumerate}[(i)]
        \item For all \(1 \leq i< j\leq n\),
        \(\Delta(E_{i,j})\in \sum\limits_{s\neq i,j}\mathbb{C}E_{s,i}+\sum\limits_{s\neq i,j}\mathbb{C}E_{s,j}+\mathbb{C}H_i+\mathbb{C}H_j\).
        \item
        For all \(1 \leq k \leq n\), \(\Delta(H_k)\in \sum\limits_{1\leq l\leq n}\mathbb{C}H_{l}\).
    \end{enumerate}
\end{lemma}

\begin{lemma}\label{lem:localD0}
Let \(\Delta\) be a local derivation on \(\mathfrak{e}(n)\) such that \(\Delta(\mathfrak{so}(n))\subseteq
H(n)\). 
\begin{enumerate}[(i)]
    \item For \(n=2m, m\geq 2\), if \(\Delta\left(\sum_{i=1}^{m}E_{2i-1,2i}\right)=0\), then \(\Delta\mid_{\mathfrak{so}(n)}=0\).
    \item For \(n=2m+1, m\geq 2\), if \(\Delta\left(\sum_{i=1}^{m}E_{2i-1,2i}\right)=0\), then \((\Delta- \lambda  
    \,\text{ad}H_{n})
\mid_{\mathfrak{so}(n)}=\) with some \(\lambda\in \mathbb{C}\).
\end{enumerate}
\end{lemma}
\begin{proof}
    By \(\Delta(\mathfrak{so}(n)\subseteq H(n)\) and  Lemma \ref{lem:4.1}, we can assume that
\begin{align}\label{ld:Eij}
\Delta(E_{i,j})=
b^{i}_{(i,j)}H_{i}+b^{j}_{(i,j)}H_{j},
\end{align}
for all \(1\leq i<j\leq n\).

(i) By \( \Delta(\sum_{i=1}^{m}E_{2i-1,2i})=0\) and \cref{ld:Eij},  we obtain 
\begin{align*}
0&=\Delta(\sum_{i=1}^{m}E_{2i-1,2i})\\
&=\sum_{i=1}^{m}(b_{(2i-1,2i)}^{2i-1}H_{2i-1}+b_{(2i-1,2i)}^{2i}H_{2i}).
\end{align*}
Then \(b_{(2i-1,2i)}^{2i-1}=b_{(2i-1,2i)}^{2i}=0\) for all \(1 \leq i
\leq m\), which deduces \[\Delta(E_{2i-1,2i})=0,  \ \text{for}\ 1 \leq i\leq m.\]

Next, we show that \(\Delta(E_{s,t})=0\) for all  \(1 \leq s<t\leq n \).
We divide the proof into four case: (1) \(s,t\) are even, (2) \(s\) is even, \(t\) is odd,  (3) \(s\) is odd, \(t\) is  even, (4) \(s,t\) are odd.  

\textbf{Case 1.} \(s,t\) are even. 

Let \(x=E_{s-1,s}+ E_{s,t}\). 
By \(\Delta(E_{s-1,s})=0\) and \cref{ld:Eij}, we have
\begin{align*}
b_{(s,t)}^{s}H_s+b_{(s,t)}^{t}H_t&=\Delta(E_{s,t})=\Delta(E_{s-1,s}+E_{s,t})\\
&=D_x(E_{s-1,s}+E_{s,t})\\
&=(\sum_{1\leq i
< j\leq n}\lambda^{x}_{i,j}\,\text{ad}E_{i,j}+\sum_{1\leq k
\leq  n}\mu^{x}_{k}\,\text{ad}H_k+\lambda_x\delta
)(E_{s-1,s}+E_{s,t})\\
&=-\mu^x_sH_{s-1}+\mu^x_{s-1}H_{s}-\mu^x_tH_{s}+\mu^x_{s}H_{t}.
\end{align*}
Then \(b^{t}_{(s,t)}=\mu_{s}^{x}=0\), which deduces that \(\Delta(E_{s,t})\in \mathbb{C}H_{s}\).

Let \(y= E_{s,t}+E_{t-1,t}\). 
By \(\Delta(E_{t-1,t})=0\) and \cref{ld:Eij}, we have
\begin{align*}
b_{(s,t)}^{s}H_s&=\Delta(E_{s,t})=\Delta(E_{s,t}+E_{t-1,t})\\
&=D_y(E_{s,t}+E_{t-1,t})\\
&=(\sum_{1\leq i
< j\leq n}\lambda^{y}_{i,j}\,\text{ad}E_{i,j}+\sum_{1\leq k
\leq  n}\mu^{y}_{k}\,\text{ad}H_k+\lambda_y\delta
)(E_{s,t}+E_{t-1,t})\\
&=-\mu^y_tH_{t-1}+\mu^y_{t-1}H_{t}-\mu^y_tH_{s}+\mu^y _{s}H_{t}.
\end{align*}
Then \(b^{s}_{(s,t)}=-\mu_{t}^{y}=0\), which deduces that \(\Delta(E_{s,t})\)=0.

\textbf{Case 2.} \(s\) is even, \(t\) is odd. 

Let \(x=E_{s-1,s}+ E_{s,t}\). 
By \(\Delta(E_{s-1,s})=0\) and \cref{ld:Eij}, we have
\begin{align*}
b_{(s,t)}^{s}H_s+b_{(s,t)}^{t}H_t&=\Delta(E_{s,t})=\Delta(E_{s-1,s}+E_{s,t})\\
&=D_x(E_{s-1,s}+E_{s,t})\\
&=(\sum_{1\leq i
< j\leq n}\lambda^{x}_{i,j}\,\text{ad}E_{i,j}+\sum_{1\leq k
\leq  n}\mu^{x}_{k}\,\text{ad}H_k+\lambda_x\delta
)(E_{s-1,s}+E_{s,t})\\
&=-\mu^x_sH_{s-1}+\mu^x_{s-1}H_{s}-\mu^x_tH_{s}+\mu^x_{s}H_{t}.
\end{align*}
Then \(b^{t}_{(s,t)}=\mu_{s}^{x}=0\), which deduces that \(\Delta(E_{s,t})\in \mathbb{C}H_{s}\).

Let \(y= E_{s,t}+E_{t,t+1}\). 
By \(\Delta(E_{t,t+1})=0\) and \cref{ld:Eij}, we have
\begin{align*}
b_{(s,t)}^{s}H_s&=\Delta(E_{s,t})=\Delta(E_{s,t}+E_{t,t+1})\\
&=D_y(E_{s,t}+E_{t,t+1})\\
&=(\sum_{1\leq i
< j\leq n}\lambda^{y}_{i,j}\,\text{ad}E_{i,j}+\sum_{1\leq k
\leq  n}\mu^{y}_{k}\,\text{ad}H_k+\lambda_y\delta
)(E_{s,t}+E_{t,t+1})\\
&=-\mu^y_tH_{s}+\mu^y _{s}H_{t}-\mu^y_{t+1}H_{t}+\mu^y_{t}H_{t+1}.
\end{align*}
Then \(b^{s}_{(s,t)}=-\mu_{t}^{y}=0\), which deduces that \(\Delta(E_{s,t})\)=0. 

\textbf{Case 3.} \(s\) is odd,\  \(t\) is even.  
If \(t=s+1\),\quad \(\Delta(E_{s,t})=\Delta(E_{s,s+1})=0\).
 Suppose that \(t>s+1\).

Let \(x=E_{s,s+1}+ E_{s,t}\). 
By \(\Delta(E_{s,s+1})=0\) and 
\cref{ld:Eij}, we have
\begin{align*}
b_{(s,t)}^{s}H_s+b_{(s,t)}^{t}H_t&=\Delta(E_{s,t})=\Delta(E_{s,s+1}+E_{s,t})\\
&=D_x(E_{s,s+1}+E_{s,t})\\
&=(\sum_{1\leq i
< j\leq n}\lambda^{x}_{i,j}\,\text{ad}E_{i,j}+\sum_{1\leq k
\leq  n}\mu^{x}_{k}\,\text{ad}H_k+\lambda_x\delta
)(E_{s,s+1}+E_{s,t})\\
&=-\mu^x_{s+1}H_{s}+\mu^x_{s}H_{s+1}-\mu^x_tH_{s}+\mu^x_{s}H_{t}.
\end{align*}
Then \(b^{t}_{(s,t)}=\mu_{s}^{x}=0\), which deduces that \(\Delta(E_{s,t})\in \mathbb{C}H_{s}\).

Let \(y= E_{s,t}+E_{t-1,t}\). 
By \(\Delta(E_{t-1,t})=0\) and \cref{ld:Eij}, we have
\begin{align*}
b_{(s,t)}^{s}H_s&=\Delta(E_{s,t})=\Delta(E_{s,t}+E_{t-1,t})\\
&=D_y(E_{s,t}+E_{t-1,t})\\
&=(\sum_{1\leq i
< j\leq n}\lambda^{y}_{i,j}\,\text{ad}E_{i,j}+\sum_{1\leq k
\leq  n}\mu^{y}_{k}\,\text{ad}H_k+\lambda_y\delta
)(E_{s,t}+E_{t-1,t})\\
&=-\mu^y_tH_{s}+\mu^y _{s}H_{t}-\mu^y_{t}H_{t-1}+\mu^y_{t-1}H_{t}.
\end{align*}
Then \(b^{s}_{(s,t)}=-\mu_{t}^{y}=0\), which deduces that \(\Delta(E_{s,t})\)=0.

\textbf{Case 4.} \(s,t\) are  odd.

Let \(x=E_{s,s+1}+ E_{s,t}\). 
By \(\Delta(E_{s,s+1})=0\) and \cref{ld:Eij}, we have
\begin{align*}
b_{(s,t)}^{s}H_s+b_{(s,t)}^{t}H_t&=\Delta(E_{s,t})=\Delta(E_{s,s+1}+E_{s,t})\\
&=D_x(E_{s,s+1}+E_{s,t})\\
&=(\sum_{1\leq i
< j\leq n}\lambda^{x}_{i,j}\,\text{ad}E_{i,j}+\sum_{1\leq k
\leq  n}\mu^{x}_{k}\,\text{ad}H_k+\lambda_x\delta
)(E_{s,s+1}+E_{s,t})\\
&=-\mu^x_{s+1}H_{s}+\mu^x_{s}H_{s+1}-\mu^x_tH_{s}+\mu^x_{s}H_{t}.
\end{align*}
Then \(b^{t}_{(s,t)}=\mu_{s}^{x}=0\), which deduces that \(\Delta(E_{s,t})\in \mathbb{C}H_{s}\).

Let \(y= E_{s,t}+E_{t,t+1}\). 
By \(\Delta(E_{t,t+1})=0\) and \cref{ld:Eij}, we have
\begin{align*}
b_{(s,t)}^{s}H_s&=\Delta(E_{s,t})=\Delta(E_{s,t}+E_{t,t+1})\\
&=D_y(E_{s,t}+E_{t,t+1})\\
&=(\sum_{1\leq i
< j\leq n}\lambda^{y}_{i,j}\,\text{ad}E_{i,j}+\sum_{1\leq k
\leq  n}\mu^{y}_{k}\,\text{ad}H_k+\lambda_y\delta
)(E_{s,t}+E_{t,t+1})\\
&=-\mu^y_tH_{s}+\mu^y _{s}H_{t}-\mu^y_{t+1}H_{t}+\mu^y_{t}H_{t+1}.
\end{align*}
Then \(b^{s}_{(s,t)}=-\mu_{t}^{y}=0\), which deduces that \(\Delta(E_{s,t})\)=0.

Overall, this completes the proof.

(ii) By the proof of (i), we have 
\[\Delta(E_{i,j})=0, \  \text{for all } \  1 \leq i<j\leq n-1.\]
Assume that \(\Delta(E_{1,n})=b^1_{(1,n)}H_1+b^n_{(1,n)}H_n\). Set \(\Delta'=\Delta+b^1_{(1,n)}\text{ad}H_n\), then \(\Delta'\) is also a local derivation on \(\mathfrak{e}(n)\) that satisfies the following: 
\[
\Delta'(\mathfrak{so}(n))\subseteq H(n),  \ \Delta'(\sum_{i=1}^{m}E_{2i-1,2i})=0, \ 
\Delta'(E_{1,n})=b_{(1,n)}^{n}H_n.
\]
Then \(\Delta'(E_{i,j})=0\) for all \(1\leq i<j\leq n-1\).

Next, we show that \(\Delta'(E_{i,n})=0\) for all \(1\leq i<n\). 
For \(i=1\), let \(x=E_{1,2}+E_{1,n}\),
then
\begin{align*}
b_{(1,n)}^{n}H_n&=\Delta'(E_{1,n})=\Delta'(E_{1,2}+E_{1,n})\\
&=D_x(E_{1,2}+E_{1,n})\\
&=(\sum_{1\leq i
< j\leq n}\lambda^{x}_{i,j}\,\text{ad}E_{i,j}+\sum_{1\leq k
\leq  n}\mu^{x}_{k}\,\text{ad}H_k+\lambda_x\delta
)(E_{1,2}+E_{1,n})\\
&=-\mu^x_2H_{1}+\mu^x _{1}H_{2}-\mu^x_{n}H_{1}+\mu^x_{1}H_{n}.
\end{align*}
This implies that \(b^{1}_{(1,n)}=\mu_{1}^{x}=0\), so \(\Delta'(E_{1,n})=0\).

For \(2\leq i\leq n-1\),  we assume that
\[
\Delta'(E_{i,n})=c_{(i,n)}^iH_i+c_{(i,n)}^nH_n.
\]
Let \(y=E_{1,n}+E_{i,n}\), then
\begin{align*}
c_{(i,n)}^{i}H_i+c_{(i,n)}^{n}H_n&=\Delta'(E_{i,n})=\Delta'(E_{1,n}+E_{i,n})\\
&=D_y(E_{1,n}+E_{i,n})\\
&=(\sum_{1\leq i
< j\leq n}\lambda^{y}_{i,j}\,\text{ad}E_{i,j}+\sum_{1\leq k
\leq  n}\mu^{y}_{k}\,\text{ad}H_k+\lambda_y\delta
)(E_{1,n}+E_{i,n})\\
&=-\mu^y_nH_{1}+\mu^y _{1}H_{n}-\mu^y_{n}H_{i}+\mu^y_{i}H_{n}.
\end{align*}
This implies that \(c^{i}_{(i,n)}=-\mu_{n}^{y}=0\), so \(\Delta'(E_{i,n})=c_{(i,n)}^nH_n\).
If \(i\) is odd, 
let \(z_1=E_{i,i+1}+E_{i,n}\), then
\begin{align*}
c_{(i,n)}^{n}H_n&=\Delta'(E_{i,n})=\Delta'(E_{i,i+1}+E_{i,n})\\
&=D_{z_1}(E_{i,i+1}+E_{i,n})\\
&=(\sum_{1\leq i
< j\leq n}\lambda^{z_1}_{i,j}\,\text{ad}E_{i,j}+\sum_{1\leq k
\leq  n}\mu^{z_1}_{k}\,\text{ad}H_k+\lambda_{z_1}\delta
)(E_{i,i+1}+E_{i,n})\\
&=-\mu^{z_1}_{i+1}H_{i}+\mu^{z_1}_{i}H_{i+1}-\mu^{z_1}_{n}H_{i}+\mu^{z_1}_{i}H_{n}.
\end{align*}
This implies that \(c^{n}_{(i,n)}=\mu_{i}^{z_1}=0\). 

If \(i\) is even, 
let \(z_2=E_{i-1,i}+E_{i,n}\), then
\begin{align*}
c_{(i,n)}^{n}H_n&=\Delta'(E_{i,n})=\Delta'(E_{i-1,i}+E_{i,n})\\
&=D_{z_2}(E_{i-1,i}+E_{i,n})\\
&=(\sum_{1\leq i
< j\leq n}\lambda^{z_2}_{i,j}\,\text{ad}E_{i,j}+\sum_{1\leq k
\leq  n}\mu^{z_2}_{k}\,\text{ad}H_k+\lambda_{z_2}\delta
)(E_{i-1,i}+E_{i,n})\\
&=-\mu^{z_2}_{i}H_{i-1}+\mu^{z_2}_{i-1}H_{i}-\mu^{z_2}_{n}H_{i}+\mu^{z_2}_{i}H_{n}.
\end{align*}
This implies that \(c^{n}_{(i,n)}=\mu_{i}^{z_2}=0\). So \(\Delta'(E_{i,n})=0\) for \( 2 \leq i \leq n-1\).

  Overall, one deduces that \(\Delta'(\mathfrak{so}(n))=(\Delta+b_{(1,n)}^1\text{ad}H_{n})(\mathfrak{so}(n))=0\), 
  as desired. 
\end{proof}
\begin{lemma}\label{lem:localD1}
    Let \(\Delta\) be a local derivation on \(\mathfrak{e}(n)\) such that \(\Delta(\mathfrak{so}(n))=0
    \). Then  \(\Delta=\lambda\,\delta
    \) for some \(\lambda\in \mathbb{C}\).
\end{lemma}
\begin{proof}
Let \(\Delta\) be a local derivation on \(\mathfrak{e}(n)\) such that \(\Delta(\mathfrak{so(n)})=0
    \).
For any \(1\leq i\leq n\), choose \(j\) such that \(1\leq i\neq j\leq n\),  then \(\Delta(E_{i,j})=0\).
Let \(x=H_i+E_{i,j}\), by Lemma \ref{lem:4.1},  we have
   \begin{align*}   \sum_{k=1}^{n}d_{i,k}H_{k}&=\Delta(H_i)=\Delta(H_i+E_{i,j})\\
&=D_x(H_i+E_{i,j})\\
&=(\sum_{1\leq s
< t\leq n}\lambda^{x}_{s,t}\,\text{ad}E_{s,t}+\sum_{1\leq l
\leq  n}\mu^{x}_{l}\,\text{ad}H_l+\lambda_x\delta
)(H_i+E_{i,j})\\
&=\sum_{1\leq s<i\leq n,s\neq j }\lambda^{x}_{s,i}
(H_s+E_{s,j})
-\sum_{1\leq i<t\leq n, t\neq j }\lambda^{x}_{i,t}(H_t+E_{t,j})
-\sum_{1\leq j<t\leq n, t\neq i}\lambda^{x}_{j,t}E_{i,t}\\
&\quad-\sum_{1\leq s<j\leq n, s\neq i}\lambda^{x}_{s,j}E_{s,i}
+\mu_i^{x}H_j-\mu_j^{x}H_i+\lambda_xH_i.
   \end{align*}
   Comparing the coefficients of \(H_k\), we get \(d_{i,k}=\lambda_{k,i}^{x}=0, \ \text{for}\ k<i\) and \(d_{i,k}=-\lambda_{i,k}^{x}=0, \ \text{for}\ i<k\neq j\leq n\). This shows that \(\Delta(H_i)\in\mathbb{C}H_i+\mathbb{C}H_j\) for all \(1\leq i\neq j\leq n\), and therefore \(\Delta(H_i)\in\mathbb{C}H_i\).

Suppose that \(\Delta(H_1)=\lambda H_1\). Let \(\Delta'=\Delta-\lambda\,\delta\). Then \[\Delta'(\mathfrak{so}(n))=0, \ \Delta'(H_1)=0.\]
For \(2 \leq i\leq n\), we assume that
\(\Delta'(H_i)=d_{i,i}'H_i\). 
Choose that \(j\notin\{ 1,i\}\),
let \(y=E_{1,i}+E_{1,j}+E_{i,j}+
H_1+H_i\), then
\begin{align*}   d_{i,i}'H_{i}&=\Delta'(H_i)=\Delta'(E_{1,i}+E_{1,j}+E_{i,j}+
H_1+H_i)\\
&=D_y(E_{1,i}+E_{1,j}+E_{i,j}+
H_1+H_i)\\
&=(\sum_{1\leq s
\neq t\leq n}\lambda^{y}_{s,t}\,\text{ad}E_{s,t}+\sum_{1\leq l
\leq  n}\mu^{y}_{l}\,\text{ad}H_l+\lambda_y\delta
)(E_{1,i}+E_{1,j}+
E_{i,j}+
H_1+H_{i})\\
&=\sum_{s\neq 1,i}\lambda_{s,1}^{y}E_{s,i}-
\sum_{s\neq 1,i}\lambda_{s,i}^{y}E_{s,1}-
\sum_{t\neq 1,i}\lambda_{1,t}^{y}E_{t,i}+
\sum_{t\neq 1,i}\lambda_{i,t}^{y}E_{t,1}
+\mu_1^{y}H_i-\mu_i^{y}H_1\\
&\quad+\sum_{s\neq 1,j}\lambda_{s,1}^{y}E_{s,j}-
\sum_{s\neq 1,j}\lambda_{s,j}^{y}E_{s,1}-
\sum_{t\neq 1,j}\lambda_{1,t}^{y}E_{t,j}+
\sum_{t\neq 1,j}\lambda_{j,t}^{y}E_{t,1}
+\mu_1^{y}H_j-\mu_j^{y}H_1\\
&\quad+\sum_{s\neq i,j}\lambda_{s,i}^{y}E_{s,j}-
\sum_{s\neq i,j}\lambda_{s,j}^{y}E_{s,i}-
\sum_{t\neq i,j}\lambda_{i,t}^{y}E_{t,j}+
\sum_{t\neq i,j}\lambda_{j,t}^{y}E_{t,i}
+\mu_i^{y}H_j-\mu_j^{y}H_i\\
&\quad+\lambda_yH_1+\sum_{1<s}\lambda_{s,1}^{y}H_s-
\sum_{1<t}\lambda_{1,t}^{y}H_t+\lambda_yH_i+\sum_{i\neq s}\lambda_{s,i}^{y}H_s-
\sum_{i\neq t}\lambda_{i,t}^{y}H_t.
   \end{align*}
 Comparing the coefficients of \(H_1,H_i,
 H_j, E_{1,j},E_{j,i}\), we see that
 \begin{align*}
     0&=-\mu_i^{y}-\mu_j^{y}+\lambda_y+\lambda_{1,i}^{y}-\lambda_{i,1}^{y}\\
      d_{i,i}'&=\mu_1^{y}-\mu_j^{y}+\lambda_y+\lambda_{i,1}^{y}-
      \lambda_{1,i}^{y}\\
0&=\mu_1^{y}+\mu_i^{y}+\lambda_{j,1}^{y}-\lambda_{1,j}^{y}+
    \lambda_{j,i}^{y}-\lambda_{i,j}^{y}\\
0&=\lambda_{j,i}^{y}- \lambda_{i,j}^{y}+\lambda_{1,i}^{y}-
    \lambda_{i,1}^{y}\\
0&=\lambda_{j,1}^{y}- \lambda_{1,j}^{y}-\lambda_{i,1}^{y}+\lambda_{1,i}^{y}.
 \end{align*}
Combinaing the above equations, we have that \(d_{i,i}'=0
\). So \(\Delta'=0\), this implies that \(\Delta=\lambda\,\delta\).
\end{proof}

Now we shall give the main result concerning local derivations on \( \mathfrak{e}(n) \).

\begin{theorem} 
\label{thm:locDer}
All local derivations on 
\( \mathfrak{e}(n) \) are derivations.
\end{theorem}
    \begin{proof}
Let \(\Delta\) be a local derivation on \(\mathfrak{e}(n)\). 
Then \(\Delta_{0,0}\mid_{\mathfrak{so}(n)}\) is a local derivation on \(\mathfrak{so}(n)\). 
In view of \cite[Theorem 3.1]{AK2},
\(\Delta_{0,0}\mid _{\mathfrak{so}(n)}\) is an inner  derivation  of  \(\mathfrak{so}(n)\), i.e.,
\(\Delta_{0,0}\mid _{\mathfrak{so}(n)}
=\text{ad}\,X\mid _{\mathfrak{so}(n)}\) for some \(X\in \mathfrak{so}(n)\). 

Set \(\Delta'=
\Delta-\text{ad}\,X\). Then \(\Delta'\) is also a local derivation
on  \(\mathfrak{e}(n)\) such that 
\(\Delta'(\mathfrak{so}(n))\subseteq H(n)\). We divide the proof into two cases: (1) \(n \) is even, (2) \(n \) is odd.

\textbf{Case 1.} \(n\) is even. Let \(u=\sum_{i=1}^{\frac{n}{2}}E_{2i-1,2i}\).
Then there exists a derivation \(D_{u}\) on \(\mathfrak{e}(n)\) such that \(\Delta' (u) = D_{u}(u)\). Set
\[
\Delta'_{1}(x) = \Delta' (x) - D_{u}(x).
\]
Then  \(\Delta'_{1}\) 
is also a
 local derivation of $\mathfrak{e}(n)$, which satisfies the condition of Lemma
\ref{lem:localD0}(i): \(\Delta'_{1}(u) = 0\). Hence,
\(\Delta'_{1}\mid_{\mathfrak{so}(n)}=0\).  By Lemma \ref{lem:localD1}, we have \(\Delta'_1=\lambda\,\delta\) for some \(\lambda\in \mathbb{C}\). Then 
\[
\Delta=\Delta'+\text{ad}\,X=\Delta_1'+D_u+\text{ad}\,X=\lambda\,\delta+D_u+\text{ad}\,X.
\]

\textbf{Case 2.} \(n\) is odd. Let \(v=\sum_{i=1}^{\frac{n-1}{2}}E_{2i-1,2i}\).
Then there exists a derivation \(D_{v}\) on \(\mathfrak{e}(n)\) such that \(\Delta'(v) = D_{v}(v)\). Set
\[
\Delta'_{2}(x) = \Delta' (x) - D_{v}(x).
\]
Then  \(\Delta'_{2}\) 
is also a
 local derivation of $\mathfrak{e}(n)$, which satisfies the condition of Lemma
\ref{lem:localD0}(ii): \(\Delta'_{2}(v) = 0\). Hence,
\(\Delta'_{2}\mid_{\mathfrak{so}(n)}=k\,\text{ad}H_{n}\mid_{\mathfrak{so}(n) }\).  By Lemma \ref{lem:localD1}, we have \(\Delta'_2-k\,\text{ad}H_{n}=l\,\delta\) for some \(l\in \mathbb{C}\). Then 
\[
\Delta=\Delta'+\text{ad}\,X=\Delta_2'+D_v+\text{ad}\,X=k\,\text{ad}\,H_n+l\,\delta+D_v+\text{ad}\,X.
\]
The proof is complete.
\end{proof}
\section*{Acknowledgments}
 The author gratefully acknowledges the financial support from the National Natural Science Foundation of China (NSFC, Grant No. 12501033) and the Guangdong Basic and Applied Basic Research Foundation (Grant No. 2024A1515110043). The author also wishes to thank Dr. Wangbin for insightful discussions and helpful comments.


\end{document}